\documentclass[a4paper,11pt,reqno]{amsart}

\usepackage{amsmath}%
\usepackage{amsfonts}%
\usepackage{amssymb}%
\usepackage{graphicx}
\usepackage{float}
\usepackage{amssymb}
\usepackage{color}
\usepackage{multirow}

\usepackage{setspace}
\usepackage{cite}

\theoremstyle{plain}
\newtheorem{acknowledgement}{Acknowledgement}

\newtheorem{theorem}{Theorem}
\newtheorem{definition}[theorem]{Definition}
\newtheorem{example}[theorem]{Example}

\newtheorem{proposition}[theorem]{Proposition}
\newtheorem{remark}[theorem]{Remark}
\newtheorem{corollary}[theorem]{Corollary}

\allowdisplaybreaks

\newcommand\NN{\mathbb{N}}

\newcommand\RR{\mathbb{R}}

\newcommand\ex{\mathbb{E}}

\newcommand\WA{\mathcal{WA}}

\newcommand\eins{\mathrm{1}}
\newcommand\lin{\mathrm{lin}}

\numberwithin{equation}{section}

\begin{document}

 \title[Optimal pointwise approximation of anticipating SDEs]{Optimal pointwise approximation of anticipating SDEs}
 
\author{Peter Parczewski}
\address{Peter Parczewski: Institute of Mathematics B6, University of Mannheim, D-68131 Mannheim, Germany}
\email{parczewski@math.uni-mannheim.de}

\date{\today}



\begin{abstract} 
We derive the
optimal rate of convergence for the mean squared error at the terminal point for anticipating linear stochastic differential equations, where the integral is interpreted in Skorohod sense.
Although alternative proof techniques are needed, our results can be seen as generalizations of the corresponding results for It\=o SDEs. As a key tool we extend optimal approximation results for vectors of correlated Wiener integrals to general random vectors, which contain the solutions of our Skorohod SDEs.
\end{abstract}
\maketitle

\section{Introduction}
We suppose a Brownian motion $(W_{t})_{t \in [0,1]}$ on the probability space $(\Omega, \mathcal{F}, P)$, where the $\sigma$-field $\mathcal{F}$ is generated by the Brownian motion and completed by null sets. Under the assumption that the Brownian motion is evaluated at an equidistant time grid, i.e. we have the information $W_{1/n}, W_{2/n},\ldots, W_1$, the investigation of optimal approximation with respect to the mean squared error (MSE) for It\=o stochastic differential equations (SDEs) is well-studied. For SDEs interpreted in the classical It\=o sense there is a rich literature on numerical results, approximation algorithms and error analysis. We mention the monographs \cite{Kloeden_Platen, GrahamTalay}. 
We also refer to the survey \cite{MGR} and the comprehensive study of the pointwise optimal approximation of It\=o SDEs given in \cite{Mueller_Gronbach}. 

Much less is known about numerical schemes for anticipating stochastic differential equations. Some upper bounds of the MSE are known from the investigation of Euler schemes for a class of Skorohod SDEs, see e.g.\cite{TorresTudor} and \cite{Shevchenko}. An Euler scheme for an anticipating Stratonovich SDE is analysed in \cite{AnhKHiga}. A Wong-Zakai result for anticipating Stratonovich SDEs is given in \cite{CoutinFrizVictoir} by rough path theory. None of these studies deal with optimal approximation or lower bounds of the MSE. Skorohod SDEs arise in several applications, e.g. the computation of derivative-free option price sensitivities \cite{Fournie, Chen_Glasserman}. 

To the best of our knowledge, this is the first study of optimal approximation for anticipating SDEs, where the integral is interpreted in Skorohod sense. 

There are very few existence results on Skorohod SDEs. However, all the difficulties arise already at linear Skorohod SDEs with a nonadapted initial value: Let $a, \sigma, f \in C^1([0,1];\RR)$ and consider the Wiener integral $I(f) = \int_{0}^{1} f(s)dW_s$ and some sufficiently smooth function $F:\RR \rightarrow \RR$ (cf. Theorem \ref{thm_OptSDELin} below). Then the unique solution $(X_t)_{t \in [0,1]}$ of the Skorohod SDE with the nonadapted initial value
\begin{align}\label{eq_SDE}
dX_t = a(t) X_t dt + \sigma(t) X_t dW_t, \quad X_0 = F(I(f)), \ t \in [0,1], 
\end{align}
exists in $L^2(\Omega \times [0,1])$ (see e.g. \cite{Buckdahn_Nualart}). 
The solution of \eqref{eq_SDE} has a simple representation in terms of a convolution operator, the Wick product, a basic tool in stochastic analysis, see e.g. \cite{Buckdahn_Nualart, Holden_Buch} and Section \ref{section_SkorohodSDE}, as
\begin{equation}\label{eq_LinSkorohodSDESol}
X_t = X_0 \diamond \exp\left(\int_{0}^{t} \sigma(s) dW_s + \int_{0}^{t} (a(s) - \sigma^2(s)/2)ds\right). 
\end{equation}
The main difficulty is the handling of the convolution operator. 
We denote the norm and inner product on $L^2 := L^2([0,1];\RR)$ by $\|\cdot \|$ and $\langle \cdot, \cdot\rangle$. The stochastic calculus on $L^2(\Omega) := L^2(\Omega,\mathcal{F},P)$ is based on the Gaussian Hilbert space $\{I(f) : f \in L^2\} \subset L^2(\Omega)$. 
The \emph{Wick exponential}, i.e. the stochastic exponential of a Wiener integral $I(f)$, is defined by
\begin{equation}\label{eq_WickExponential}
\exp^{\diamond}(I(f)) := \exp\left(I(f) - \|f\|^2/2\right). 
\end{equation}

Given a random variable $X \in L^2(\Omega)$, we are interested in the approximation 
$$\widehat{X}^n \in L^2(\Omega,\sigma(W_{1/n}, \ldots, W_1),P),
$$
that minimizes the mean squared error (MSE)
$
\ex[(X-\widehat{X}^n)^2]
$.
This is clearly given by
\begin{equation}\label{eq_OptimalAppproxDef}
\widehat{X}^n := \ex[X|W_{1/n}, W_{2/n}, \ldots, W_1].  
\end{equation}

In the following we mean by 
$$
f' \in BV
$$ 
that $f \in L^2$ is differentiable and $f': [0,1] \rightarrow \RR$ is of bounded variation. 
As a generalization of Wick exponentials following \cite{Buckdahn_Nualart}, we define the class of \emph{Wick-analytic functionals $\WA$} as
\begin{equation*}
F(I(f_1), \ldots, I(f_K)) = \left(\sum\limits_{k=0}^{\infty}a_{1,k} I(f_1)^{\diamond k}\right) \diamond \cdots \diamond \left(\sum\limits_{k=0}^{\infty} a_{K,k} I(f_K)^{\diamond k}\right),
\end{equation*}
where $K \in \NN$, $f_1, \ldots, f_K \in L^2$ and $\sup\{\sqrt[k]{k! |a_{i,k}|} : i\leq K, k\geq 1\}  < \infty$. The smoothness of Wick-analytic functionals $F(\cdot)$ is explained in Section \ref{section_SkorohodSDE}.

Our main result on optimal pointwise approximation is:

\begin{theorem}\label{thm_OptSDELin}
Suppose $f', \sigma' \in BV$, $a: \RR \rightarrow \RR$ is integrable, $X_0 = F(I(f)) \in \WA$. Then for the solution of the Skorohod SDE \eqref{eq_SDE} we have
$$
\lim_{n \rightarrow \infty} n^2 \, \ex[(X_1-\widehat{X_1}^n)^2]= \frac{e^{2\int_{0}^{1} a(s) ds}}{12} \int_{0}^{1} \ex\left[\left(\left(f'(s) F'(I(f)) + \sigma'(s) F(I(f))\right)\diamond e^{\diamond I(\sigma)}\right)^2\right] ds.
$$
\end{theorem}

A direct application of Theorem \ref{thm_OptSDELin} on a Wick exponential type initial value gives:

\begin{example}\label{exe_LinSDEWickExp}
Suppose $f', \sigma' \in BV$, $a$ is integrable and the nonadapted initial value $X_0 = e^{\diamond}(I(f))$. The solution of the linear Skorohod SDE is then given by
$$ 
X_t = e^{\diamond I(f)} \diamond e^{\diamond \int_{0}^{t} \sigma(s) dW_s} e^{\int_{0}^{t} a(s) ds} = e^{\diamond I(f + \eins_{[0,t)}\sigma)} e^{\int_{0}^{t} a(s) ds}.
$$
and the terminal value $X_1$ satisfies the asymptotic optimal approximation
$$
\lim_{n \rightarrow \infty} n^2 \, \ex[(X_1-\widehat{X_1}^n)^2] =
 e^{2\int_{0}^{1} a(s) ds+\|f+\sigma\|^2} \|f' + \sigma'\|^2/12.
$$
However, here $X_1$ coincides with the solution of a linear It\=o SDE and the constant above is contained in the optimal approximation results in \cite{Mueller_Gronbach}. Considering some $X_0 = e^{\diamond I(f)}$ with $f \in C^1([0,2];\RR), f\geq 1$ and $\mathrm{supp} (f) = [0,2]$ (in that case a larger $L^2(\Omega)$ as well), this is not covered by \cite{Mueller_Gronbach} anymore, as $X_{1}$ is now nonadapted. The extension of Theorem \ref{thm_OptSDELin} to such extended time horizons and initial values is straightforward.
\end{example}


In contrast to It\=o SDEs, concerning Skorohod integrals and anticipating SDEs, we have no martingale or Markov tools, in particular we have to handle the lack of It\=o isometry and well-known bounds as martingale inequalities (e.g. the BDG inequality). Therefore we make use of subtle computations of the Wiener chaos expansion of all objects involved. As already observed on the optimal approximation of Skorohod integrals \cite{NP}, this necessary alternative approach leads surprisingly to natural generalizations of the corresponding results for It\=o SDEs. Notice again that Theorem \ref{thm_OptSDELin} extends the corresponding results for linear It\=o SDEs in \cite[Theorem 1]{Mueller_Gronbach}.

The main tools for our considerations are optimal approximation results for functionals of Wiener integrals. This leads to nice compatibility relations for optimal approximations of all random elements (Section \ref{section_OptWienerChaos}). Due to these general results, Theorem \ref{thm_OptSDELin} extends easily to more general nonadapted initial values (Theorem \ref{thm_OptSDELinMulti}).

The generalization to an asymptotically optimal approximation scheme for \eqref{eq_SDE} is part of subsequent work.

\section{Skorohod SDEs and Wick-analytic functionals}\label{section_SkorohodSDE}

For a possibly nonadapted process $(u_s)_{s \in [0,1]}$, the Skorohod integral $\int_{0}^{1} u_s dW_s$ can be defined as a natural extension of the It\=o integral, see e.g. \cite{ DiNunno, Holden_Buch, Pardoux}. 

We make use of a definition via Wick exponentials \eqref{eq_WickExponential}, which have many useful properties. In particular, 
$
\{\exp^{\diamond}(I(f)) : f \in L^2\}$
is a total set in $L^p(\Omega,\mathcal{F},P)$, $p>0$ (see e.g. \cite[Cor. 3.40]{Janson}) and allows the characterization of random variables, the S-transform definition of the Skorohod integral (cf. e.g. \cite[Sec. 16.4]{Janson}):

\begin{definition}\label{def_SkorohodIntegral}
Suppose $u=(u_s)_{s \in [0,1]}$ is a (possibly nonadapted) square integrable process on $(\Omega, \mathcal{F}, P)$ and $Y \in L^2(\Omega, \mathcal{F}, P)$ such that
\[
\forall g \in L^2: \ \ex[Y e^{\diamond I(g)}] = \int_{0}^{1} \ex[u_s e^{\diamond I(g)}] g(s) ds, 
\]
then $\int_{0}^{1}u_s dW_{s} = Y$ defines the Skorohod integral of $u$ with respect to $(W_{t})_{t \in [0,1]}$.
\end{definition}

For more information on the Skorohod integral we refer to \cite{Janson, Kuo, Nualart}.
The definition above is closely related to a convolution imitating the product of uncorrelated random variables as $\ex[X \diamond Y]=\ex[X]\ex[Y]$, which is implicitly contained in the Skorohod integral and a fundamental tool in stochastic analysis. Due to the injectivity of $X \mapsto \ex[X e^{\diamond I(g)}]$  the \emph{Wick product} can be introduced via
\begin{equation*}
\forall g \in L^2 \ : \ \ex[(X \diamond Y) e^{\diamond I(g)}] = \ex[X e^{\diamond I(g)}] \ex[ Y e^{\diamond I(g)}] 
\end{equation*}
on a dense subset in $L^{2}(\Omega) \times L^{2}(\Omega)$ (see e.g. \cite[Chap. 16]{Janson} for more details). For example, it is  
$
e^{\diamond I(g)} \diamond e^{\diamond I(h)} = e^{\diamond I(g+h)}
$
for all $f,g \in L^2$.
In particular, for a Wiener integral $I(f)$, the Hermite polynomials play the role of monomials in standard calculus as $
(I(f))^{\diamond k} = h^k_{\|f\|^2}(I(f))$ and the notation Wick exponential is well justified by $\exp^{\diamond}(I(f)) = \sum_{k=0}^{\infty} \frac{1}{k!} I(f)^{\diamond k}$. For more details on Wick exponentials we refer to \cite{Holden_Buch, Janson, Kuo}. We note that the derivative rule for Hermite polynomials as polynomials $h^k_{\|f\|^2}(x) = p(x,\|f\|^2)$ gives for the ordinary derivative $\dfrac{\partial}{\partial x} h^k_{\|f\|^2}(I(f)) = k h^{k-1}_{\|f\|^2}(I(f))$.

An example in $\WA$, for the Wiener integral 
$\int_{0}^{1}W_s \, ds = \int_{0}^{1}(1-s)dW_s$, is 
$$
\sin\left(\int_{0}^{1}W_s \, ds\right) = \sin^{\diamond}\left(\int_{0}^{1}W_s \, ds\right)e^{\int_{0}^{1}(1-s)^2\,ds/2} =  \sum_{k=1}^{\infty} \frac{(-1)^{k-1}e^{1/6}}{(2k-1)!}\left(\int_{0}^{1}W_s \, ds\right)^{\diamond (2k-1)},
$$
which cannot be simulated exactly.

We denote the linear span by $\lin(\WA)$. Notice that $\lin (\WA) \subset L^2(\Omega)$ (see e.g. \cite[Proposition 9]{NP}).

\begin{remark}\label{rem_derivativeWA}
Let a Wick-analytic functional $F(I(f)) = \sum_{k=0}^{\infty} a_k I(f)^{\diamond k}$. Thanks to the derivative rule for Hermite polynomials, it is $F'(I(f)) = \sum_{k \geq 0} (k+1)a_{k+1} I(f)^{\diamond k}$. Due to $(k+1)/k\leq 2$, we have
$
\sup \{\sqrt[k]{(k+1)! |a_{k+1}|} : k\geq 1\} \leq \left(\sup \{\sqrt[k]{k! |a_{k}|} : k\geq 1\}\right)^2<\infty.
$ 
Therefore it clearly is $F'(I(f)) \in \WA$.
\end{remark}

An iteration of the conclusion in Remark \ref{rem_derivativeWA} (cf. \cite[Proposition 10]{NP}) yields:

\begin{proposition}\label{prop_DerivativesWA}
All derivatives of elements in $\lin (\WA)$ as in Remark \ref{rem_derivativeWA} are in $\lin (\WA)$ as well.
\end{proposition}
One can identify $\lin (\WA)$ as a class of smooth random variables in Malliavin calculus (cf. \cite[p. 25]{Nualart}). However, for optimal approximation the Wick-analytic representation is more appropriate, see Section \ref{section_OptWienerChaos}.


\section{Approximation and Wiener chaos}\label{section_OptWienerChaos}

In this section we present general results on optimal approximation and on simple implementable approximations for the class $\lin (\WA)$. However, besides these essential tools for our main result (Theorem \ref{thm_OptSDELin}), these optimality results and optimality constants are interesting for its own. At the end we prove the main result Theorem \ref{thm_OptSDELin} and give a further generalization.

The Wiener chaos expansion in terms of Wick analytic functionals has the advantage that the optimal approximation carries over to functionals in terms of Wick products. In fact, we have (see \cite[Corollary 9.4]{Janson} or \cite[Lemma 6.20]{DiNunno}):

\begin{proposition}\label{prop_WickConditionalExpectation}
For $X,Y, X\diamond Y \in L^2(\Omega)$ and the sub-$\sigma$-field $\mathcal{G} \subseteq \mathcal{F}$:
$$
\ex[X\diamond Y|\mathcal{G}] = \ex[X|\mathcal{G}]\diamond \ex[Y|\mathcal{G}].
$$ 
\end{proposition}

This immediately implies
that the computations on optimal approximation can be extended to Wick-analytic functionals on the underlying Wiener integrals:

\begin{corollary}\label{cor_WickCarriesOverOpt}
For $F= F(I(f_1), I(f_2), \ldots, I(f_K)) \in \lin (\WA)$ and the optimal approximation \eqref{eq_OptimalAppproxDef}, we have
$$
\widehat{F}^n = F\left(\widehat{I(f_1)}^{n}, \widehat{I(f_2)}^{n}, \ldots, \widehat{I(f_K)}^{n}\right).
$$
\end{corollary}
 
However, we are interested in the implications for the MSE and the constants involved. This is done in three subsections: Finite chaos random variables, Infinite chaos random variables and the conclusion to the proof of our main result.

\subsection{Finite chaos}\label{subsection_FiniteChaos}
We clearly have 
$$
\widehat{W_t}^n = \ex[W_t|W_{1/n}, W_{2/n}, \ldots, W_1] = W_t^{\lin},
$$
where 
$$
W_t^{\lin} := W_{i/n} + n(t-i/n)(W_{(i+1)/n}-W_{i/n}), \quad t \in [i/n, (i+1)/n),
$$
is the linear interpolation of $W$ with respect to the equidistant time grid.
All further computations will be based on the following convergences:

\begin{proposition}\label{proposition_WienerIntCov}
Suppose $f, 'g' \in BV$. Then 
$$
\lim_{n \rightarrow \infty} n^2 \; \ex\left[(I(f) - \widehat{I(f)}^n)(I(g) - \widehat{I(g)}^n)\right] = \frac{1}{12}\langle f', g' \rangle.
$$
\end{proposition}

\begin{proof}
Via integration by parts $I(f) = f(1) W_1 - \int_{0}^{1} f'(s)W_s \, ds$. Hence, Fubini's theorem yields
\begin{equation}\label{eq_WienerIntDiff}
I(f) - \widehat{I(f)}^n = - \int_{0}^{1} f'(s) \left(W_s - W^{\lin}_s\right) ds = -\sum_{i0}^{n-1} \int_{i/n}^{(i+1)/n} f'(s) \left(W_s - W^{\lin}_s\right) ds.
\end{equation}
We recall the well-known covariances of these Brownian bridges:
\begin{equation*}
\int_{i/n}^{(i+1)/n} \int_{j/n}^{(j+1)/n} \ex[(W_s - W^{\lin}_s)(W_t - W^{\lin}_t)] ds\, dt = \eins_{\{i=j\}}\frac{1}{12 n^3}. 
\end{equation*}
Hence, for $g,h \in C([0,1];\RR)$, the mean value theorem then gives
$$
\int_{i/n}^{(i+1)/n} \int_{j/n}^{(j+1)/n} g(s) h(t)\ex[(W_s - W^{\lin}_s)(W_t - W^{\lin}_t)] ds\, dt = g(\xi_i)h(\overline{\xi}_i)\eins_{\{i=j\}}\frac{1}{12 n^3}
$$
for appropriate $\xi_i, \overline{\xi}_i \in [i/n,(i+1)/n]$. Thus, \eqref{eq_WienerIntDiff}, gives the Riemann sum
\begin{align}\label{eq_RiemannSumFirst}
\ex\left[(I(f) - \widehat{I(f)}^n)(I(g) - \widehat{I(g)}^n)\right]
=\frac{1}{12 n^2} \left(\frac{1}{n}\sum_{i = 0}^{n-1} f'(\xi_i) g'(\overline{\xi}_i) \right).
\end{align}
We denote $h_k := h(k/n)$ for $h \in C([0,1];\RR)$ and the total variation 
\begin{align*}
T(h) := \sup\left\{\sum_{i=0}^{m-1}|h(t_{i+1})-h(t_i)| : m \in \NN, 0=t_0<\ldots <t_{m}=1\right\}. 
\end{align*}
Hence, via the mean value theorem for some $n\int_{i/n}^{(i+1)/n} h(s) ds = h(s_i), s_i \in \left[i/n, (i+1)/n\right]$, it is
\begin{equation}\label{eq_RiemannSumRateBV}
\left|\int_{0}^{1} h(s) ds - \frac{1}{n}\sum_{i=0}^{n-1} h_i\right| \leq 
\frac{1}{n}\sum_{i=0}^{n-1} \left|h(s_i) - h(i/n)\right| \leq\frac{T(h)}{n}.
\end{equation}
Moreover, for all $i=0,\ldots, n-1$, we notice
$$
\left|f'_i g'_i - f'(\xi_i)g'(\overline{\xi}_i)\right| \leq \left|f'_i - f'(\xi_i) \right|\|g'\|_{\infty} + \left|g'_i - g'(\overline{\xi}_i)\right|\|f'\|_{\infty},
$$
and therefore, via triangle inequality,
\begin{equation}\label{eq_ErrorAsRiemannSum}
\frac{1}{12}\left|\frac{1}{n}\sum_{i=0}^{n-1} f'_i g'_i - \left(\frac{1}{n}\sum_{i = 0}^{n-1} f'(\xi_i) g'(\overline{\xi_i}) \right)\right| \leq \frac{1}{12 n}\left(T(f')\|g'\|_{\infty} + T(g')\|f'\|_{\infty}\right). 
\end{equation}
Thanks to \eqref{eq_RiemannSumFirst} - \eqref{eq_ErrorAsRiemannSum},
we conclude
\begin{align}\label{eq_WienerIntOACovariance}
\lim_{n \rightarrow \infty} n^2\, \ex\left[(I(f) - \widehat{I(f)}^n)(I(g) - \widehat{I(g)}^n)\right] &= \frac{1}{12} \langle f', g' \rangle.
\end{align}
\end{proof}

\begin{remark}\label{remark_WienerInt}
In the simplest case, both converge towards $\frac{1}{12}\|f'\|^2$ as $n$ tends to infinity:
$$
n^2\, \ex\left[(I(f) - \widehat{I(f)}^n)(I(f) - I^n(f))\right], \quad n^2\, \ex\left[(I(f) - \widehat{I(f)}^n)I(f)\right].
$$

As a direct consequence of \eqref{eq_RiemannSumRateBV}, \eqref{eq_ErrorAsRiemannSum} and $T(f' g') \leq T(f')\|g'\|_{\infty} + T(g')\|f'\|_{\infty}$, we observe the error expansion
$$
\left|n^2\, \ex\left[(I(f) - \widehat{I(f)}^n)(I(g) - \widehat{I(g)}^n)\right] - \frac{1}{12}\langle f',g'\rangle \right| \leq \frac{1}{6n}\left(T(f')\|g'\|_{\infty}+T(g')\|f'\|_{\infty}\right),
 $$ 
 which can be extended to interesting error expansions for all further results.
\end{remark}

The multiple chaos extension of the covariance limits in Proposition \ref{proposition_WienerIntCov} is:
\begin{proposition}\label{proposition_OptWienerIntWick}
Suppose $f', g' \in BV$ and $k \in \NN$ is fixed. Then 
$$
\lim_{n \rightarrow \infty} n^2 \; \ex\left[(I(f)^{\diamond k} - (\widehat{I(f)}^n)^{\diamond k}) (I(g)^{\diamond k} - (\widehat{I(g)}^n)^{\diamond k})\right] = \frac{1}{12}\langle f', g' \rangle\, k\, k!\, \langle f, g \rangle^{(k-1)}.
$$
\end{proposition}

\begin{proof}
 For higher chaos terms we observe the standard expansion
\begin{equation}\label{eq_WickPowerExp}
I(f)^{\diamond k} - (\widehat{I(f)}^n)^{\diamond k} = (I(f) - \widehat{I(f)}^n) \diamond \sum_{j=1}^{k} (\widehat{I(f)}^n)^{\diamond j-1} \diamond I(f)^{\diamond k-j}. 
\end{equation}
We show that the right hand side is close enough to the simplified variable
$$
(I(f) - \widehat{I(f)}^n) \diamond k I(f)^{\diamond (k-1)}.
$$
Dealing with $L^2$-norms of Gaussian variables, we will frequently make use of Wick's Theorem,
\begin{equation}\label{eq_WickTheorem}
\ex\left[\left(I(f_1) \diamond \cdots \diamond I(f_n)\right) \left(I(g_1) \diamond \cdots \diamond I(g_m)\right)\right] = \eins_{\{n=m\}} \sum\limits_{\sigma \in \mathcal{S}_{n}} \prod\limits_{i=1}^{n} \langle f_i, g_{\sigma(i)}\rangle,
\end{equation}
for all $n,m \in \NN$, $f_1,\ldots, f_n, g_1,\ldots, g_m \in L^2$, where $\mathcal{S}_{n}$ denotes the group of permutations on $\{1, \ldots, n\}$ (see e.g. \cite[Theorem 3.9]{Janson}). 
Via \eqref{eq_WickPowerExp} and the reformulation for general products
\begin{align*}
\sum_{j=2}^{k}a^{k-j}\left(a^{j-1} - b^{j-1}\right) &= (a-b)\sum_{j=2}^{k}a^{k-j}\sum_{l=1}^{j-1} a^{l-1} b^{j-1-l} = (a-b)\sum_{j=2}^{k}\sum_{l=1}^{j-1} a^{k-1-l} b^{l-1}\\
&= (a-b)\sum_{l=1}^{k-1}(k-l) a^{k-1-l} b^{l-1}, 
\end{align*}
for the difference, we have
\begin{align}\label{eq_WickPowerDiffExpansion}
&(I(f) - \widehat{I(f)}^n) \diamond k I(f)^{\diamond (k-1)} - \left(I(f)^{\diamond k} - (\widehat{I(f)}^n)^{\diamond k}\right)\nonumber\\
&=  \left(I(f) - \widehat{I(f)}^n\right) \diamond \sum_{j=2}^{k} I(f)^{\diamond k-j} \diamond \left(I(f)^{\diamond (j-1)} -  (\widehat{I(f)}^n)^{\diamond (j-1)}\right)\nonumber\\
&=   \left(I(f) - \widehat{I(f)}^n\right)^{\diamond 2} \diamond \sum_{l=1}^{k-1}(k-l) I(f)^{\diamond (k-1-l)} \diamond (\widehat{I(f)}^n)^{\diamond (l-1)}.
\end{align}
We give a sufficient upper bound on the $L^2$-norm. By the covariances
\begin{align}
\ex\left[(I(f) - \widehat{I(f)}^n)(I(g) - \widehat{I(g)}^n)\right] &=\ex\left[(I(f) - \widehat{I(f)}^n)I(g)\right],
\ &\ex\left[(I(f) - \widehat{I(f)}^n)\widehat{I(g)}^n\right] &= 0,\nonumber\\
\ex\left[I(f)\widehat{I(f)}^n\right] &= \ex\left[I(f)^2\right]\geq \ex\left[(\widehat{I(f)}^n)^2\right]\label{eq_covariances}
\end{align}
and Wick's Theorem \eqref{eq_WickTheorem}, for all $0 \leq m,m' <k$, according to the scheme of numbers of factors
\[
\begin{array}{r c| c l}
2 \times \left\{ \right.&I(f) - \widehat{I(f)}^n &I(f) - \widehat{I(f)}^n & \left. \right\} \times 2\\
m\times \left\{ \right. & I(f) & I(f) & \left. \right\} \times m'\\
k-m\times \left\{ \right. & \widehat{I(f)}^n & \widehat{I(f)}^n &\left. \right\} \times k-m'
\end{array}
\]
and the shorthand notation 
$$
e_n := \ex\left[(I(f) - \widehat{I(f)}^n)^2\right],
$$
we obtain
\begin{align}\label{eq_HorribleWickProductNorms}
&
\ex\left[\left((I(f) - \widehat{I(f)}^n)^{\diamond 2} \diamond I(f)^{\diamond m} \diamond (\widehat{I(f)}^n)^{\diamond k-m}\right)\left((I(f) - \widehat{I(f)}^n)^{\diamond 2} \diamond I(f)^{\diamond m'} \diamond (\widehat{I(f)}^n)^{\diamond k-m'}\right)\right]\nonumber\\
&\leq 2 \, e_n^2\, k!\|f\|^{2k} + 4m m' \; e_n^3\, (k-1)!\|f\|^{2(k-1)}+m (m-1) m' (m'-1) \; e_n^4\, (k-2)!\|f\|^{2(k-2)}.
\end{align}
Thus, via \eqref{eq_WickPowerDiffExpansion}--\eqref{eq_HorribleWickProductNorms} and $\sum_{l=1}^{k-1-m}\left((k-l)\cdots (k-l-m)\right) = \frac{1}{2+m}k(k-1)\cdots (k-1-m)$ for $k-1-m\geq 0$ (clear by induction),
we conclude
\begin{align}\label{eq_WickPowerDiffSimpler}
&\ex\left[\left(\left(I(f)^{\diamond k} - (\widehat{I(f)}^n)^{\diamond k}\right) -  (I(f) - \widehat{I(f)}^n) \diamond k I(f)^{\diamond (k-1)} \right)^2\right] \nonumber\\
&= \ex\left[\left(\sum_{l=1}^{k-1}(k-l) \left(I(f) - \widehat{I(f)}^n\right)^{\diamond 2} \diamond I(f)^{\diamond (k-1-l)} \diamond (\widehat{I(f)}^n)^{\diamond (l-1)}\right)^2\right] \nonumber\\
&= \sum_{l,l'=1}^{k-1}(k-l)(k-l') \ex\left[\left((I(f) - \widehat{I(f)}^n)^{\diamond 2} \diamond I(f)^{\diamond (k-1-l)} \diamond (\widehat{I(f)}^n)^{\diamond l-1}\right)\right.\nonumber\\
&\hspace{4cm} \times \left.\left((I(f) - \widehat{I(f)}^n)^{\diamond 2} \diamond I(f)^{\diamond k-i-l'} \diamond (\widehat{I(f)}^n)^{\diamond l'-1}\right)\right] \nonumber\\
&\leq 2\left(\sum_{l=1}^{k-1}(k-l)\right)^2e_n^2 (k-2)!\|f\|^{2(k-2)} + 4\left(\sum_{l=1}^{k-2}(k-l)(k-l-1)\right)^2e_n^3 (k-3)!\|f\|^{2(k-3)}\nonumber\\
&\quad + \left(\sum_{l=1}^{k-3}(k-l)(k-l-1)(k-l-2)\right)^2e_n^4 (k-4)!\|f\|^{2(k-4)}\nonumber\\
&= \frac{1}{2}\left(k(k-1)\right)^2e_n^2 \ex\left[\left(I(f)^{\diamond (k-2)}\right)^2\right] + \frac{4}{9}\left(k(k-1)(k-2)\right)^2e_n^3 \ex\left[\left(I(f)^{\diamond (k-3)}\right)^2\right]\nonumber\\
& \quad + \frac{1}{16}\left(k(k-1)(k-2)(k-3)\right)^2e_n^4 \ex\left[\left(I(f)^{\diamond (k-4)}\right)^2\right].
\end{align}
In particular, via Proposition \ref{proposition_WienerIntCov}
\begin{equation}\label{eq_WickPowerDiffSimplerLandau}
\ex\left[\left((I(f)^{\diamond k} - (\widehat{I(f)}^n)^{\diamond k}) -  (I(f) - \widehat{I(f)}^n) \diamond k I(f)^{\diamond (k-1)} \right)^2\right]  \in  \mathcal{O}(n^{-4}).
\end{equation}
Therefore, by $AB-ab = (A-a)B + a(B-b)$ and the Cauchy-Schwarz inequality
\begin{align}\label{eq_WickPowerDiffSimplerLandau2}
&\left|\ex\left[\left(I(f)^{\diamond k} - (\widehat{I(f)}^n)^{\diamond k}\right)\left(I(g)^{\diamond k} - (\widehat{I(g)}^n)^{\diamond k}\right)\right]\right.\nonumber\\
&\left. -\ex\left[\left((I(f) - \widehat{I(f)}^n) \diamond k I(f)^{\diamond (k-1)} \right)\left((I(g) - \widehat{I(g)}^n) \diamond k I(g)^{\diamond (k-1)} \right)\right]\right| \in  \mathcal{O}(n^{-4}).
\end{align}
Hence, due to \eqref{eq_WickPowerDiffSimplerLandau2}, Wick's Theorem \eqref{eq_WickTheorem}, the covariances \eqref{eq_covariances} and Proposition \ref{proposition_WienerIntCov}, we conclude
\begin{align}\label{eq_WickConvFiniteChaos}
&\lim_{n \rightarrow \infty} n^2 \; \ex\left[\left(I(f)^{\diamond k} - (\widehat{I(f)}^n)^{\diamond k}\right)\left(I(g)^{\diamond k} - (\widehat{I(g)}^n)^{\diamond k}\right)\right]\nonumber\\ 
&= \lim_{n \rightarrow \infty} n^2 \; \ex\left[\left((I(f) - \widehat{I(f)}^n) \diamond k I(f)^{\diamond (k-1)} \right)\left((I(g) - \widehat{I(g)}^n) \diamond k I(g)^{\diamond (k-1)} \right)\right]\nonumber\\
&= \lim_{n \rightarrow \infty} n^2 \; \ex\left[(I(f) - \widehat{I(f)}^n)(I(g) - \widehat{I(g)}^n)\right] k\, k!\, \langle f, g \rangle^{(k-1)}\nonumber\\
&\quad + \lim_{n \rightarrow \infty} n^2 \; \ex\left[(I(f) - \widehat{I(f)}^n)(I(g) - \widehat{I(g)}^n)\right]^2 k^4\, (k-2)! \,\langle f, g \rangle^{(k-2)}\nonumber\\
&= \frac{1}{12}\langle f', g' \rangle k\, k!\, \langle f, g \rangle^{(k-1)}.
\end{align}
\end{proof}

\subsection{Infinite chaos}\label{subsection_InfiniteChaos}

The paradigmatic result on optimal approximation is:
 
\begin{theorem}\label{thm_OptWickFunc}
Suppose $f', g' \in BV$, $F=F(I(f)), G=G(I(g)) \in \WA$. Then 
$$
\lim_{n \rightarrow \infty} n^2\, \ex\left[(F-\widehat{F}^n) (G-\widehat{G}^n)\right]= \frac{1}{12} \langle f', g'\rangle \, \ex\left[F'(I(f)) G'(I(g))\right].
$$
\end{theorem}

\begin{proof}
We firstly present the proof for $F=G$. The Wiener chaos expansion yields $F(I(f)) = \sum_{k\geq 0} \frac{a_k}{k!} I(f)^{\diamond k}$ for some unique coefficients $a_k \in \RR$.
Due to Proposition \ref{prop_DerivativesWA}, $\dfrac{\partial^k}{\partial x^k}F(I(f)) \in L^2(\Omega)$ for all $k \in \NN$.
Thanks to the derivative rule for Hermite polynomials $\dfrac{\partial}{\partial x} I(f)^{\diamond k} = k I(f)^{\diamond k-1}$, we observe for all $j \in \NN$,
\begin{align*}
\ex\left[\left(\dfrac{\partial^j}{\partial x^j}F(I(f))\right)^2\right] 
&= \sum_{k\geq j} \frac{a_k^2}{(k-j)!^2} \ex\left[\left(I(f)^{\diamond (k-j)}\right)^2\right].
\end{align*}
Hence, via \eqref{eq_WickPowerDiffSimpler} and the shorthand notation $e_n = \ex\left[(I(f) - \widehat{I(f)}^n)^2\right]$, we conclude
\begin{align}\label{eq_WickPowerFuncDiffSimpler}
&\sum_{k\geq 1} \frac{a_k^2}{k!^2}\ex\left[\left((I(f)^{\diamond k} - (\widehat{I(f)}^n)^{\diamond k}) -  (I(f) - \widehat{I(f)}^n) \diamond k I(f)^{\diamond (k-1)} \right)^2\right] \nonumber\\
&\leq \frac{1}{2}e_n^2 \sum_{k\geq 2} \frac{a_k^2}{(k-2)!^2} \ex\left[\left(I(f)^{\diamond (k-2)}\right)^2\right] + \frac{4}{9}e_n^3 \sum_{k\geq 3} \frac{a_k^2}{(k-3)!^2} \ex\left[\left(I(f)^{\diamond (k-3)}\right)^2\right]\nonumber\\
&\quad + \frac{1}{16}e_n^4 \sum_{k\geq 4} \frac{a_k^2}{(k-4)!^2} \ex\left[\left(I(f)^{\diamond (k-4)}\right)^2\right]\nonumber\\
&= \frac{e_n^2}{2}\ex\left[\left(F''\right)^2\right] + \frac{4e_n^3}{9}\ex\left[\left(\dfrac{\partial^3}{\partial x^j}F\right)^2\right] + \frac{e_n^4}{16}\ex\left[\left(\dfrac{\partial^4}{\partial x^j}F\right)^2\right]\in \mathcal{O}(n^{-4}).
\end{align}
Thus, via Corollary \ref{cor_WickCarriesOverOpt}, \eqref{eq_WickPowerFuncDiffSimpler} and Proposition \ref{proposition_OptWienerIntWick} (cf. \eqref{eq_WickConvFiniteChaos}), we obtain
\begin{align*}
\lim_{n \rightarrow \infty} n^2 \; \ex\left[\left(F - \widehat{F}^{n}\right)^2\right] &= \lim_{n \rightarrow \infty} n^2 \; \sum_{k\geq 1} \frac{a_k^2}{k!^2}\ex\left[\left(I(f)^{\diamond k} - (\widehat{I(f)}^n)^{\diamond k}\right)^2\right] \nonumber\\
&=\lim_{n \rightarrow \infty} n^2 \; \sum_{k\geq 1} \frac{a_k^2}{k!^2}\ex\left[\left((I(f) - \widehat{I(f)}^n) \diamond k I(f)^{\diamond (k-1)} \right)^2\right] \nonumber\\
&= \frac{1}{12} \|f'\|^2\sum_{k\geq 1} \frac{a_k^2}{(k-1)!}  \|f\|^{2(k-1)} = \frac{1}{12} \|f'\|^2 \ex\left[\left(F'(I(f))\right)^2\right].
\end{align*}
The proof for $F\neq G$ with the Wiener chaos expansion $G(I(g)) = \sum_{k\geq 0} \frac{b_k}{k!} I(g)^{\diamond k}$ proceeds analogously, as via \eqref{eq_WickPowerDiffSimplerLandau2}, \eqref{eq_WickPowerFuncDiffSimpler} and \eqref{eq_WickConvFiniteChaos},
\begin{align*}
 &\lim_{n \rightarrow \infty} n^2 \; \ex\left[\left(F - \widehat{F}^{n}\right)\left(G - \widehat{G}^{n}\right)\right]\\
 &= \lim_{n \rightarrow \infty} n^2 \; \sum_{k\geq 1} \frac{a_k b_k}{k!^2}\ex\left[\left(I(f)^{\diamond k} - (\widehat{I(f)}^n)^{\diamond k}\right)\left(I(g)^{\diamond k} - (\widehat{I(g)}^n)^{\diamond k}\right)\right] \nonumber\\
&=\lim_{n \rightarrow \infty} n^2 \; \sum_{k\geq 1} \frac{a_k b_k}{k!^2}\ex\left[\left((I(f) - \widehat{I(f)}^n) \diamond k I(f)^{\diamond (k-1)} \right)\left((I(g) - \widehat{I(g)}^n) \diamond k I(g)^{\diamond (k-1)} \right)\right] \nonumber\\
&= \frac{1}{12} \langle f', g'\rangle \sum_{k\geq 1} \frac{a_k b_k}{(k-1)!}  \langle f,g\rangle^{k-1} = \frac{1}{12} \langle f', g'\rangle \ex\left[F'(I(f)) G'(I(g))\right].
\end{align*}
\end{proof}

 \begin{example}
Suppose $f' \in BV$. Then we have
$$
\lim_{n \rightarrow \infty} n^2 \; \ex\left[\left(e^{I(f)}-\widehat{e^{I(f)}}^n\right)^2\right] = \frac{1}{12} \|f'\|^2 e^{2\|f\|^2}.
$$  
 \end{example}

\begin{remark}\label{rem_OptErrorExpansion}
Looking at \eqref{eq_HorribleWickProductNorms}, we observe that $\ex[(\widehat{I(f)}^n)^2] \leq \ex[(I(f))^2]$ is the only reason for the inequality. Hence, asymptotically, we obtain equalities in \eqref{eq_HorribleWickProductNorms}, \eqref{eq_WickPowerDiffSimpler} and \eqref{eq_WickPowerFuncDiffSimpler}. This yields the following expansion of the optimal approximation error with $e_n := \ex\left[(I(f) - \widehat{I(f)}^n)^2\right]$,
\begin{align*}
\ex\left[\left(F(I(f)) - \widehat{F(I(f))}^{n}\right)^2\right] &\sim e_n\, \ex\left[\left(F'(I(f))\right)^2\right]  + \frac{e_n^2}{2}\, \ex\left[\left(F''(I(f))\right)^2\right]\nonumber\\
&\quad + \frac{4e_n^3}{9}\, \ex\left[\left(F^{(3)}(I(f))\right)^2\right] + \frac{e_n^4}{16}\, \ex\left[\left(F^{(4)}(I(f))\right)^2\right].
\end{align*}
\end{remark}

Theorem \ref{thm_OptWickFunc} can be extended in various directions to multivariate functionals. For the proofs of our main result we need:

\begin{theorem}\label{thm_OptWickFuncMult}
Suppose 
$f', g', f'_1, g'_1,\ldots,\in BV$
and define the abbreviations
\begin{align*}
F &= F(I(f_1), \ldots, I(f_m)), \quad G= G(I(g_1), \ldots, I(g_m)), \quad 
F_{x_k} := \dfrac{\partial}{\partial x_k} F.
\end{align*}
Then 
$$
\lim_{n \rightarrow \infty} n^2\, \ex\left[(F-\widehat{F}^n) (G-\widehat{G}^n)\right]= \frac{1}{12} \sum_{i,j=1}^{m}\langle f'_i, g'_j\rangle\, \ex\left[F_{x_i}G_{x_j}\right].
$$

\end{theorem} 


\begin{proof}
Due to various Wiener chaos expansions the notations for arbitrary functionals become easily elaborately. However, the proof is a straightforward extension of the arguments for Theorem \ref{thm_OptWickFunc}. We present the proof for the two-dimensional case for $G=F=F(I(f), I(g)) \in \WA$. All other cases are straightforward generalizations as before. Let the Wiener chaos expansion
$$
F(I(f), I(g)) = \sum_{k,l=0}^{\infty} a_{k,l} I(f)^{\diamond k} \diamond I(g)^{\diamond l} 
$$
for some coefficients $a_{k,l} \in \RR$. Via Corollary \ref{cor_WickCarriesOverOpt}, it is
\begin{align}\label{eq_OAMulti1}
&\ex\left[\left(F(I(f), I(g)) - \widehat{F(I(f), I(g))}^{n}\right)^2\right]\nonumber\\
&= \sum_{\substack{k,l, k', l' \geq 0\\ k+l=k'+l'>0}} a_{k,l} a_{k',l'} \ex\left[\left(I(f)^{\diamond k} \diamond I(g)^{\diamond l} - (\widehat{I(f)}^{n})^{\diamond k} \diamond (\widehat{I(g)}^{n})^{\diamond l} \right)\right.\nonumber\\
&\hspace{4cm} \left. \times\left(I(f)^{\diamond k'} \diamond I(g)^{\diamond l'} - (\widehat{I(f)}^{n})^{\diamond k'} \diamond (\widehat{I(g)}^{n})^{\diamond l'} \right)\right].
\end{align}
As 
$$
I(f)^{\diamond k} \diamond I(g)^{\diamond l} - (\widehat{I(f)}^{n})^{\diamond k} \diamond (\widehat{I(g)}^{n})^{\diamond l} = (I(f)^{\diamond k} -  (\widehat{I(f)}^{n})^{\diamond k}) \diamond I(g)^{\diamond l} +  (\widehat{I(f)}^{n})^{\diamond k} \diamond (I(g)^{\diamond l} -  (\widehat{I(g)}^{n})^{\diamond l}),
$$ 
the right hand side in \eqref{eq_OAMulti1} is reduced to covariances of the terms of the type
\begin{equation*}
(I(f)^{\diamond k} -  (\widehat{I(f)}^{n})^{\diamond k}) \diamond I(g)^{\diamond l}, \quad (I(f)^{\diamond k} -  (\widehat{I(f)}^{n})^{\diamond k}) \diamond (\widehat{I(g)}^{n})^{\diamond l}.
\end{equation*}
Analogously to the proof of Proposition \ref{proposition_OptWienerIntWick} (see e.g. \eqref{eq_WickPowerDiffSimplerLandau}--\eqref{eq_WickPowerDiffSimplerLandau2}), these terms behave in covariance computations like
$$
\left(I(f) - \widehat{I(f)}^{n}\right) \diamond k I(f)^{\diamond k} \diamond I(g)^{\diamond l}.
$$
Suppose $f', g', \bar{f}', \bar{g}' \in BV$. We recall that $I(f)^{\diamond k} \diamond I(g)^{\diamond l}$ are polynomials $p(x_1,x_2)$ (of $(x_1,x_2)=I(f), I(g)$) and therefore differentiable. Then, analogously to \eqref{eq_WickPowerDiffSimpler} and due to the derivative rule for Hermite polynomials, for $k+l=k'+l'$ we obtain
\begin{align*}
&\ex\left[\left((I(f)^{\diamond k} -  (\widehat{I(f)}^{n})^{\diamond k}) \diamond I(g)^{\diamond l}\right)\left((I(\bar{f})^{\diamond k'} -  (\widehat{I(\bar{f})}^{n})^{\diamond k'}) \diamond I(\bar{g})^{\diamond l'}\right)\right]\\
&\sim \ex\left[\left((I(f) -  \widehat{I(f)}^n) \diamond k I(f)^{\diamond k-1} \diamond I(g)^{\diamond l}\right) \left( ((I(\bar{f}) -  \widehat{I(\bar{f})}^n) \diamond k'I(\bar{f})^{\diamond k'-1} \diamond I(\bar{g})^{\diamond l'}\right)\right]\\
&\sim \frac{1}{12n^2}\langle f', \bar{f}'\rangle \, \ex\left[\left(k I(f)^{\diamond k-1} \diamond I(g)^{\diamond l}\right) \left( (k'I(\bar{f})^{\diamond k'-1} \diamond I(\bar{g})^{\diamond l'}\right)\right]\\
&= \frac{1}{12n^2}\langle f', \bar{f}'\rangle \, \ex\left[\left(\dfrac{\partial}{\partial x_1} I(f)^{\diamond k} \diamond I(g)^{\diamond l}\right) \left(\dfrac{\partial}{\partial x_1} I(\bar{f})^{\diamond k'} \diamond I(\bar{g})^{\diamond l'}\right)\right].
\end{align*}
Thus, as in Proposition \ref{proposition_OptWienerIntWick}, for a fixed chaos (of order $N \in \NN$), we conclude
\begin{align*}
&\sum_{\substack{k,l, k', l' \geq 0,\ k+l=k'+l' = N}} a_{k,l} a_{k',l'} \ex\left[\left(I(f)^{\diamond k} \diamond I(g)^{\diamond l} - (\widehat{I(f)}^{n})^{\diamond k} \diamond (\widehat{I(g)}^{n})^{\diamond l} \right)\right.\\
&\hspace{5cm} \left. \times\left(I(f)^{\diamond k'} \diamond I(g)^{\diamond l'} - (\widehat{I(f)}^{n})^{\diamond k'} \diamond (\widehat{I(g)}^{n})^{\diamond l'} \right)\right]\\
&\sim \frac{1}{12n^2} \|f'\|^2 \sum
a_{k,l} a_{k',l'} \ex\left[\left(\dfrac{\partial}{\partial x_1} I(f)^{\diamond k} \diamond I(g)^{\diamond l}\right)\left(\dfrac{\partial}{\partial x_1} I(f)^{\diamond k'} \diamond I(g)^{\diamond l'}\right)\right]\\
&\quad + \frac{1}{12n^2} \|g'\|^2 \sum
a_{k,l} a_{k',l'} \ex\left[\left(\dfrac{\partial}{\partial x_2} I(f)^{\diamond k} \diamond I(g)^{\diamond l}\right)\left(\dfrac{\partial}{\partial x_2} I(f)^{\diamond k'} \diamond I(g)^{\diamond l'}\right)\right]\\
&\quad + \frac{2}{12n^2} \langle f', g\rangle\sum
a_{k,l} a_{k',l'} \ex\left[\left(\dfrac{\partial}{\partial x_1} I(f)^{\diamond k} \diamond I(g)^{\diamond l}\right)\left(\dfrac{\partial}{\partial x_2} I(f)^{\diamond k'} \diamond I(g)^{\diamond l'}\right)\right].
\end{align*}
Hence, analogously to Theorem \ref{thm_OptWickFunc}, summing up (Fubini's theorem applies as we have uniform bounds via Proposition \ref{prop_DerivativesWA}) leads to the asserted asymptotics
\begin{align}\label{eq_OAMultiLast}
&\ex\left[\left(F(I(f), I(g)) - \widehat{F(I(f), I(g))}^{n}\right)\right]\nonumber\\
&\sim \frac{1}{12n^2} \|f'\|^2\, \ex\left[F_{x_1}^2\right] + \frac{1}{12n^2} \|g'\|^2\, \ex\left[F_{x_2}^2\right] + \frac{2}{12n^2} \langle f', g' \rangle\, \ex\left[F_{x_1}F_{x_2}\right].
\end{align}
\end{proof}

\begin{remark}\label{remark_InfChaosGeneral}
In particular, due to 
$$
\ex\left[\left(F+G - \widehat{F+G}^n\right)^2\right] = \ex\left[(F-\widehat{F}^n)^2 + (G-\widehat{G}^n)^2 + 2(F-\widehat{F}^n)(G-\widehat{G}^n)\right], 
$$
or polarization, we conclude Theorem \ref{thm_OptSDELinMulti}  for $F, G \in \lin(\WA)$ as well. 
\end{remark}

\subsection{Proof of the main result} 

Our optimal approximation result is now an easy application:

\begin{proof}[Proof of Theorem \ref{thm_OptSDELin}]
Due to \eqref{eq_LinSkorohodSDESol}, Proposition \ref{prop_WickConditionalExpectation} and the deterministic term $\int_{0}^{1}a(s) ds$, we have
\begin{align*}
\ex[(X_1 - \widehat{X_1}^n)^2] &= \ex\left[\left(F(I(f)) \diamond e^{\diamond I(\sigma)} e^{\int_{0}^{1}a(s) ds} - \widehat{F(I(f))}^n \diamond e^{\diamond \widehat{I(\sigma)}^n} e^{\int_{0}^{1}a(s) ds}\right)^2\right]\\
&= e^{2\int_{0}^{1}a(s) ds}\, \ex\left[\left(F(I(f)) \diamond e^{\diamond I(\sigma)} - \widehat{F(I(f))}^n \diamond e^{\diamond \widehat{I(\sigma)}^n} \right)^2\right].
\end{align*}
For the random variable 
$$
G(I(f), I(\sigma)) := F(I(f)) \diamond e^{\diamond I(\sigma)},
$$

by $F(I(f)) \in \WA$ and Proposition \ref{prop_DerivativesWA}, these derivatives exist in $L^2(\Omega)$ (for the function $G(x_1,x_2)$):
$$
G_{x_1} = \dfrac{\partial}{\partial x_1} G(I(f), I(\sigma)) =  F'(I(f)) \diamond e^{\diamond I(\sigma)}, \qquad G_{x_2} = \dfrac{\partial}{\partial x_2} G(I(f), I(\sigma))= F(I(f)) \diamond e^{\diamond I(\sigma)}.
$$
Thanks to Theorem \ref{thm_OptWickFuncMult}, we therefore conclude
\begin{align*}
&\lim_{n \rightarrow \infty} n^2\, \ex\left[\left(F(I(f)) \diamond e^{\diamond I(\sigma)} - \widehat{F(I(f))} \diamond e^{\diamond \widehat{I(\sigma)}} \right)^2\right] = \frac{1}{12} \left(\|f'\|^2 \ex[(F'(I(f)) \diamond e^{\diamond I(\sigma)})^2]\right.\\
&\qquad \left. + \|\sigma'\|^2 \ex[(F(I(f)) \diamond e^{\diamond I(\sigma)})^2] +2\langle f',\sigma'\rangle \ex[(F'(I(f)) \diamond e^{\diamond I(\sigma)})(F(I(f)) \diamond e^{\diamond I(\sigma)})]\right)\\
&\qquad = \int_{0}^{1} \ex\left[\left(\left(f'(s) F'(I(f)) + \sigma'(s) F(I(f))\right)\diamond e^{\diamond I(\sigma)}\right)^2\right] ds.
\end{align*}
This yields the asserted optimal convergence in Theorem \ref{thm_OptSDELin}. 
\end{proof}

We obtain easily further multivariate generalizations of the nonadapted initial value:

\begin{theorem}\label{thm_OptSDELinMulti}
Suppose $f'_1, \ldots, f'_m, \sigma' \in BV$, $a: \RR \rightarrow \RR$ is integrable and 
$$
X_0 := F = F(I(f_1), \ldots, I(f_m))\in \lin (\WA).
$$
Then for the solution of the Skorohod SDE \eqref{eq_SDE} we have
$$
\lim_{n \rightarrow \infty} n^2 \, \ex[(X_1-\widehat{X_1}^n)^2] = \frac{ e^{2\int_{0}^{1} a(s) ds}}{12}\left(\int_{0}^{1} \ex\left[\left(\left( \sum_{k=1}^{m}f'_k(s)F_{x_k} + \sigma'(s) F\right)\diamond e^{\diamond I(\sigma)}\right)^2\right] ds\right).
$$
\end{theorem}

\begin{proof}
Concerning a multivariate initial value $F(I(f_1), \ldots, I(f_m))$ in Theorem \ref{thm_OptSDELinMulti}, we conclude analogously via Theorem \ref{thm_OptWickFuncMult}, Remark \ref{remark_InfChaosGeneral} and the function 
$$
G\left(I(f_1), \ldots, I(f_m), I(\sigma)\right) := F(I(f_1), \ldots, I(f_m)) \diamond e^{\diamond I(\sigma)} \in \lin(\WA).
$$
\end{proof}

\begin{remark}
The investigation of an asymptotically optimal approximation scheme for the linear Skorohod SDEs in Theorems \ref{thm_OptSDELin}, \ref{thm_OptSDELinMulti}  will follow in a subsequent work. 
\end{remark}

\begin{acknowledgement}
The author thanks Andreas Neuenkirch for many fruitful discussions.  
\end{acknowledgement}


\begin{thebibliography}{plain}


\bibitem{AnhKHiga} Ahn, H. and Kohatsu-Higa, A.
The Euler scheme for anticipating stochastic differential equations. \textit{Stochastics Stochastics Rep.} \textbf{54} (3-4), (1995) 247--269.




\bibitem{Buckdahn} Buckdahn, R. Skorohod stochastic differential equations of diffusion type. \textit{Probab. Theory Related Fields} \textbf{93} (3), (1992) 297--323. 

\bibitem{Buckdahn_Nualart} Buckdahn, R. and Nualart, D. Linear stochastic differential equations and Wick products. \textit{Probab. Theory Related Fields} \textbf{99} (4), (1994) 501--526.

\bibitem
{Chen_Glasserman}
Chen, N. and Glasserman, P. Malliavin Greeks without Malliavin calculus. \textit{Stochastic Process. Appl.} \textbf{117} (11), 1689-1723, (2007). 


\bibitem{CoutinFrizVictoir} Coutin, L. and Friz, P. and Victoir, N.
Good rough path sequences and applications to anticipating stochastic calculus.
\textit{Ann. Probab.} \textbf{35} (3), (2007) 1172--1193. 

\bibitem{DiNunno} Di Nunno, G. and {\O}ksendal, B. and Proske, F. \textit{Malliavin calculus for L\'{e}vy processes with applications to finance} Universitext. Springer, Berlin, 2009.

\bibitem{Fournie} Fourni\'e, E. and Lasry, J.-M. and Lebuchoux, J. and Lions, P.-L. and Touzi, N. Applications of Malliavin calculus to Monte Carlo methods in finance. \textit{Finance Stoch.} \textbf{3} (4), (1999) 391--412.

\bibitem{GrahamTalay} Graham, C. and Talay, D. \textit{Stochastic Simulation and Monte Carlo Methods}, Springer, Berlin, 2013.

\bibitem{Holden_Buch} Holden H. and {\O}ksendal, B. and Ub{\o}e, J. and Zhang, T.
\textit{Stochastic Partial Differential Equations. A Modeling, White Noise Functional Approach. Second Edition} Springer, New York, 2010.


\bibitem{Janson} Janson, S.
\textit{Gaussian Hilbert Spaces.}, Cambridge University Press, Cambridge, 1997.

\bibitem{Kloeden_Platen} Kloeden, P. and Platen, E. \textit{Numerical solution of stochastic differential equations.} Applications of Mathematics, 23. Springer-Verlag, Berlin, 1992.

\bibitem{Kuo} Kuo, H.-H.
\textit{White Noise Distribution Theory.}
Probability and Stochastics Series. CRC Press, Boca Raton, 1996.


\bibitem{MGR} M{\"u}ller-Gronbach, T. and Ritter, K. Minimal errors for strong and weak approximation of stochastic differential equations. In \textit{Monte Carlo and Quasi-Monte Carlo Methods 2006}, Springer, 53--82, 2008.

\bibitem{Mueller_Gronbach} M\"uller-Gronbach, T. Optimal pointwise approximation of SDEs based on Brownian motion at discrete points. \textit{Ann. Appl. Probab.} \textbf{14} (4), (2004) 1605--1642.

\bibitem{NP} Neuenkirch, A. and Parczewski, P. Optimal approximation of Skorohod integrals. \textit{J. Theor. Probab.} \textbf{31} (1), (2018) 206--231.

\bibitem{Nualart} Nualart, D.
\textit{The Malliavin Calculus and Related Topics.} Second Edition. Probability and its Applications. Springer, New York, 2006. 




\bibitem{Pardoux} Pardoux, \'E. Applications of anticipating stochastic calculus to stochastic differential equations. In \textit{Lecture Notes in Math.} \textbf{1444}, Springer, (1990) 63--105.


\bibitem{Shevchenko} Shevchenko, G. Euler approximation for anticipating stochastic quasilinear differential equations. 
\textit{Theory Probab. Math. Statist.} \textbf{72}, (2006) 167--175. 

\bibitem{TorresTudor} Torres, S. and Tudor, C.
The Euler scheme for a class of anticipating stochastic differential equations.
\textit{Random Oper. Stochastic Equations} \textbf{12} (3), (2004) 211--224.  

\end{thebibliography}
\end{document}